\documentclass[11pt,twosides]{amsart}
\usepackage{latexsym,amsfonts,amssymb,amsthm,mathrsfs,amsmath,amscd,eufrak,amscd, color,MnSymbol}

\theoremstyle{plain}
   \newtheorem{teo}{Theorem}
   
   \newtheorem{lema}[teo]{Lemma}

\theoremstyle{definition}
   
\theoremstyle{remark}

\numberwithin{equation}{section}

\newcommand{\bmo}{\mathrm{BMO}}
\newcommand{\N}{\mathbb{N}} % Conjunto N
\newcommand{\Z}{\mathbb{Z}} % Conjunto Z
\newcommand{\R}{\mathbb{R}} % Conjunto R
\newcommand{\s}[3]{\sum_{#1}^{#2}{#3}} % Sumatoria desde ... hasta... de ...
\newcommand{\normbmo}[1]{\|#1\|_{\mathrm{BMO}}} % Norma BMO de ...
\newcommand{\norm}[1]{\|#1\|} % Norma de ...

\newcommand{\normbmows}[1]{\|#1\|_{\mathrm{BMO}_w^*}} % Norma BMO de ...
\newcommand{\normexpl}[1]{\|#1\|_{\mathrm{exp}L,Q}}
\newcommand{\normexplw}[1]{\|#1\|_{\mathrm{exp}L,Q,w}}
\newcommand{\normexplv}[3]{\|#1\|_{\mathrm{exp}L,#3,#2}}
\newcommand{\normlloglv}[3]{\|#1\|_{L\mathrm{log}L,#3,#2}}
\newcommand{\supp}{\mathrm{supp}}

\allowdisplaybreaks

\begin{document}

\title[Mixed weak estimates of Sawyer type for commutators]{Mixed weak estimates of Sawyer type for commutators of singular integrals and related operators}

% Information for first author
\author[F. Berra]{Fabio Berra}
\address{CONICET and Departamento de Matem\'{a}tica (FIQ-UNL),  Santa Fe, Argentina.}
\email{fberra@santafe-conicet.gov.ar}

%Information for second author
\author[M. Carena]{Marilina Carena}
\address{CONICET (FIQ-UNL) and Departamento de Matem\'{a}tica (FHUC-UNL),  Santa Fe, Argentina.}
\email{marilcarena@gmail.com}

%Information for third author
\author[G. Pradolini]{Gladis pradolini}
\address{CONICET and Departamento de Matem\'{a}tica (FIQ-UNL),  Santa Fe, Argentina.}
 \email{gladis.pradolini@gmail.com}

\thanks{The authors were supported by CONICET, UNL and ANPCyT}



%%% ----------------------------------------------------------------------

\begin{abstract} We study mixed weak type inequalities for the commutator
$[b,T]$, where $b$ is a BMO function and $T$ is a Calder\'on-Zygmund
operator. More precisely, we prove that for every  $t>0$
\begin{equation*}%\label{tesis_teo2.2}
uv\left(\left\{x\in\R^n:
\left|\frac{[b,T](fv)(x)}{v(x)}\right|>t\right\}\right)\leq
C\int_{\R^n}\phi\left(\frac{|f(x)|}{t}\right)u(x)v(x)\,dx,
\end{equation*}
where $\phi(t)=t(1+\log^{+}{t})$,  $u\in A_1$ and $v\in
A_{\infty}(u)$. Our technique involves the classical
Calder\'on-Zygmund decomposition, which allow us to give a direct
proof. We use this result to prove an analogous inequality for
higher order commutators. We also obtain a mixed estimation for a
wide class of maximal operators associated to certain Young
functions of $L\log L$ type which are in intimate relation with the
commutators. This last estimate involves an arbitrary weight $u$ and a
radial function $v$ which is not even locally integrable.
\end{abstract}

%%% ----------------------------------------------------------------------
\maketitle
%%% ----------------------------------------------------------------------

\section*{Introduction}

In \cite{CU-M-P} the authors considered weighted weak type norm inequalities given by
\begin{equation}\label{eq: mixed}
uv\left(\left\{x\in\R^n: \frac{T(fv)(x)}{v(x)}>t\right\}\right)\leq \frac{C}{t}\int_{\R^n}|f(x)|u(x)v(x)\,dx, \hspace{1truecm} t>0
\end{equation}
for some positive constant $C$, where $T$ is either the Hardy-Littlewood  maximal operator or any Calder\'on-Zygmund
operator. The authors proved that \eqref{eq: mixed} holds if $u,v$ are weights such that $u,v\in A_1$, or $u\in A_1$ and $v\in
A_{\infty}(u)$. This result proves the conjecture given by Sawyer in \cite{Sawyer}, where \eqref{eq: mixed} is proved in $\mathbb R$
for the Hardy-Littlewood  maximal operator $M$ and $u,v\in A_1$. The author also conjectured that the inequality  holds if $T$ is
the Hilbert transform. The motivation of Sawyer for consider \eqref{eq: mixed}
yields a new proof of the classical Muckenhoupt's
Theorem  concerning to the boundedness of $M$ in $L^p(w)$, for $1<p<\infty$ and $w\in A_p$. Indeed, given $w\in A_p$, from the P. Jones
factorization Theorem  we have that $w=uv^{1-p}$, with $u,v\in A_1$, so that the operator $S(f)=M(fv)/v$ is bounded on $L^\infty(uv)$.
Hence, the Muckenhoupt's Theorem is obtained from the usual Marcinkiewicz interpolation Theorem provided that $S$ is of weak type
 $(1,1)$ with respect to the measure $uvdx$, which is precisely \eqref{eq: mixed} with $T=M$ (see \cite{Stein_Weiss}).

In this paper  we study inequalities of the type described in
\eqref{eq: mixed} for higher order commutators of Calder\'on-Zygmund
operators with BMO symbols, generalizing the results obtained in
\cite{CU-M-P}. However, our techniques are quite different of those
given in this article. As far as we know, this type of estimates are
new even for the case of the first order commutator.
\\
We also obtain an analogous mixed estimation for generalized maximal
operators associated to higher order
commutators which are defined by means of a Young function. This estimate extends the results given in \cite{O-P} to a wide class of maximal
 operators involving Luxemburg averages.\\
There is a close relationship between the boundedness properties of
commutators acting on different functional spaces and partial
differential equations, and it is well known that the continuity
properties of such operators provides us with regular solutions of
certain PDE's. Several authors were working in this direction, (see,
for example \cite{BCM}, \cite{BCe}, \cite{ChFL}, \cite{ChFL2},
\cite{DR} and \cite{Rios} between a vast amount of articles).
Therefore, it seems appropriate to explore the weighted inequalities
for that operators and, particularly, we shall be concerned with the
mixed estimates mentioned above.

Recall that a linear operator $T$ is a \emph{Calder\'on-Zygmund operator}  if $T$ is bounded on $L^2(\mathbb R^n)$ and there exists a standard kernel $K$
such that for $f\in L^2$ with compact support,
\[Tf(x)=\int_{\mathbb R^n} K(x,y)f(y)\,dy,\quad\quad x\notin \supp f.\]
We say that $K:\mathbb R^n\times \mathbb R^n\backslash \Delta\to\mathbb C$ is a \emph{standard kernel} if it satisfies a size condition given by %there exists $\delta>0$ such that
\[|K(x,y)|\leq \frac{C}{|x-y|^n},\]
and the smoothness conditions
\begin{equation}\label{eq:prop del nucleo}
|K(x,y)-K(x,z)|\leq C\frac{|x-z|}{|x-y|^{n+1}},\quad \textrm{ if } |x-y|>2|y-z|,
\end{equation}
\[|K(x,y)-K(w,z)|\leq C\frac{|x-w|}{|x-y|^{n+1}},\quad \textrm{ if } |x-y|>2|x-w|.\]
Recall, as well, that the \emph{commutator operator} $[b,T]$ is formally defined, for adequate functions $f$, by
\[[b,T]f=bT(f)-T(bf).\]

We are now in position to state our main result.
\bigskip
\begin{teo}\label{teo: main}
Let $u,v$ be weights such that $u\in A_1$ and $v\in A_{\infty}(u)$.
Let $T$ be any Calder\'on-Zygmund
operator  and let
$b\in$BMO. Then, for every  $t>0$ we have that
\bigskip
\begin{equation*}%\label{tesis_teo2.2}
uv\left(\left\{x\in\R^n:
\left|\frac{[b,T](fv)(x)}{v(x)}\right|>t\right\}\right)\leq
C\int_{\R^n}\Phi\left(\normbmo{b}\frac{|f(x)|}{t}\right)u(x)v(x)\,dx,
\end{equation*}
\bigskip
where $\Phi(t)=t(1+\log^{+}{t})$.
\end{teo}
The theorem above is a starting point to prove by induction
an analogous inequality for higher order commutators, denoted  by $T^m_b$, for a non negative integer $m$ and defined by induction as follows:
%\[T^m_b=\underbrace{[b,\dots,[b,T]]}_{(m \textrm{ times})}.\]
 $T^0_b=T$ and $T^m_b=[b, T^{m-1}_b]$ for $m\geq1$. If $x$ is not in the support of $f$ then it is clear that
\[T^m_bf(x)=\int\left(b(x)-b(y)\right)^m K(x-y)f(y)\,dy.\]
Thus we obtain the following result.
%In \cite{Carlos-sharp} is proved that if $w\in A_q$ then $T_b^m:L^q(w)\hookrightarrow L^q(w)$ for every $m\in\N$.
\bigskip

\begin{teo}\label{teo: orden m}
Let $u,v$ be weights such that $u\in A_1$ and $v\in A_{\infty}(u)$.
Let $T$ be any Calder\'on-Zygmund
operator and let
$b\in$ BMO. Then, for every  $t>0$ and every positive integer $m$ we have that
\bigskip
\begin{equation}\label{tesis_teo2.2}
uv\left(\left\{x\in\R^n:
\left|\frac{T_b^m(fv)(x)}{v(x)}\right|>t\right\}\right)\leq
C\int_{\R^n}\Phi_m\left(\normbmo{b}^m\frac{|f(x)|}{t}\right)u(x)v(x)\,dx,
\end{equation}
\bigskip
where $\Phi_m(t)=t(1+\log^{+}{t})^m$.
\end{teo}
%for every  $t>0$ and every positive integer $m$ we have that
%\begin{equation}\label{tesis_teo2.2}
%uv\left(\left\{x\in\R^n:
%\left|\frac{T_b^m(fv)(x)}{v(x)}\right|>t\right\}\right)\leq
%C\int_{\R^n}\Phi_m\left(\normbmo{b}^m\frac{|f(x)|}{t}\right)u(x)v(x)\,dx,
%\end{equation}
%where $\Phi_m(t)=t(1+\log^{+}{t})^m$.

Observe that since $\Phi_m$ is submultiplicative, that is, $\Phi_m(ab)\leq C\Phi_m(a)\Phi_m(b)$ for $ab\geq 0$, we have that \eqref{tesis_teo2.2}
 implies that\\

\[uv\left(\left\{x\in\R^n:
\left|\frac{T_b^m(fv)(x)}{v(x)}\right|>t\right\}\right)\leq
C \Phi_m\left(\normbmo{b}^m\right)\int_{\R^n}\Phi_m\left(\frac{|f(x)|}{t}\right)u(x)v(x)\,dx.\]\\

When $m=0$, that is $T_b^0=T$, the same estimate holds and our proof also works for this case. This estimation was first proved in \cite{CU-M-P}. However, our proof is quite different from that one since we shall not use the control of $T$ by the Hardy-Littlewood maximal operator, but the proof is straightforwardly related with the classical Calder\'on-Zygmund decomposition.\\

We can relax the hypotheses on the weights in both theorems above in order to obtain mixed inequalities for other operators. For example, in \cite{O-P} the authors give a mixed estimation for the Hardy-Littlewood maximal operator on $\R^n$ for the case in which $u$ is a weight and $v$ is a power that is not even locally integrable. We wonder if an analogous estimation holds for $M^2$. Indeed, we have proved a more general result, involving the operator $M_{\Phi}$ for the case $\Phi(t)=t^r(1+\log^+t)^\delta$, with $r\geq 1$ and $\delta\geq 0$. Observe that, when $r=1$ and $\delta=1$ this is the desired result, since it is well known that $M^2$ is equivalent to $M_{L\mathrm{log}L}$ (see next section).

The result that we obtain is the following.
\bigskip
\begin{teo}\label{teo: caso m_phi}
Let $u$ be a weight and $v(x)=|x|^\beta$, where $\beta<-n$. Define $w(x)=1/\Phi(v^{-1}(x))$, where $\Phi(t)=t^r(1+\log^+t)^{\delta}$ with $r\geq1$ and $\delta\geq 0$. Then, for every $t>0$,
\bigskip
\begin{equation}\label{caso m phi}
uw\left(\left\{x\in\R^n: \frac{M_{\Phi}(fv)}{v}>t\right\}\right)\leq C\int_{\R^n}\Phi\left(\frac{fv}{t}\right)Mu(x)\,dx.
\end{equation}
\end{teo}

By using the submultiplicativity of $\Phi$ the formula given in \eqref{caso m phi} can be rewritten as follows
\bigskip
\[\left\|\frac{M_{\Phi}(fv)}{v}\right\|_{\widetilde{L^\Psi}(uw)}:=\sup_{t>0}\Psi(t)\hspace{0.2truecm}uw\left(\left\{x:\frac{M_{\Phi}(fv)}{v}>t \right\}\right)\leq C\int_{\R^n}\Phi(fv)Mu(x)\,dx,\]
\bigskip
where $\Psi(t)=1/\Phi(1/t)$ and $\|f\|_{\widetilde{L^\Psi}}$ denotes the weak Orlicz norm associated to $\Psi$.\\

Let us observe that if $u\in A_1$ then we have that
\bigskip
\[uw\left(\left\{x\in\R^n: \frac{M_{\Phi}(fv)}{v}>t\right\}\right)\leq C\int_{\R^n}\Phi\left(\frac{fv}{t}\right)u(x)\,dx.\]
\bigskip

On the other hand, if
 $\Phi(t)=t$, we get $M_{\Phi}=M$ and $w=v$, and thus we obtain the same estimation given in \cite{O-P}.

%\begin{teo}\label{teo: caso m2}
%Let $u$ be an $A_1$ weight and $v=|x|^\beta$, where $\beta<-n$. Define $w=\frac{1}{\Phi(v^{-1})}$, where $\Phi(t)=t(1+\log^+(t))$. Then for all $t>0$,
%\[uw\left\{x\in\R^n: \frac{M^2(fv)}{v}>t\right\}\leq C\int_{\R^n}\Phi\left(\frac{fv}{t}\right)u(x)\,dx,\]
%where $f$ is bounded and have compact support.
%\end{teo}

%CUBOS DILATADOS
%

\section{Preliminaries and definitions}

 Let us recall that a \emph{weight} $w$  is a locally integrable  function defined on
 $\R^n$, such that $0<w(x)<\infty$ a.e. $x\in \R^n$.
 For $1<p<\infty$ the \emph{Muckenhoupt $A_p$ class} is defined as the set of all weights $w$
 for which there exists a positive constant $C$ such that the inequality
\[\left(\frac{1}{|Q|}\int_Q w\right)\left(\frac{1}{|Q|}\int_Q w^{-\frac{1}{p-1}}\right)^{p-1}
\leq C\] holds for every cube $Q\subset \R^n$,  with sides parallel
to the coordinate axes.  For $p=1$, we say that $w\in A_1$ if there
exists a positive constant $C$ such that
 \begin{equation*}
  \frac{1}{|Q|}\int_Q w\le C\, \inf_Q w(x),
 \end{equation*}
 for every cube $Q\subset \R^n$. %Notice that we can consider cubes
% $Q$ with edges parallel to the coordinate axes, instead of balls $B$.
The smallest constant $C$ for which the Muckenhoupt condition holds
is called the $A_p$-constant of $w$, and denoted by $[w]_{A_p}$.
The $A_\infty$ class is defined by the collection of all the $A_p$ classes.
It is easy to see
that if $p<q$ then $A_p\subseteq A_q$. Given $1<p<\infty$, we use
$p'$ to denote the conjugate exponent $p/(p-1)$.
For $p=1$ we take $p'=\infty$. Some classical references for the basic theory of Muckenhoupt weights
 are for example~\cite{javi} and~\cite{garcia-rubio}.\\

An important property of  Muckenhoupt weights is the \emph{reverse H\"{o}lder's condition}. This means that given
$w\in A_p$, for some $1\leq p<\infty$, there exists a positive constant $C$
and $s>1$ that depends only on the dimension $n$, $p$ and
$[w]_{A_p}$, such that for every cube $Q$
\begin{equation*}
\left(\frac{1}{|Q|}\int_Q w^s(x)\,dx\right)^{1/s}\leq
\frac{C}{|Q|}\int_Q w(x)\,dx.
\end{equation*}
We write $w\in
\textrm{RH}_s$ to point out that the inequality above holds,
and we denote by $[w]_{\textrm{RH}_s}$ the smallest constant $C$ for
which this condition holds. A weight $w$ belongs to RH$_\infty$ if
there exists a positive constant $C$ such
that\begin{equation*}
\sup_Q w\leq\frac{C}{|Q|}\int_Q w,
\end{equation*}
for every $Q\subset \R^n$. Let us observe that
$\textrm{RH}_\infty\subseteq \textrm{RH}_s\subseteq \textrm{RH}_q$,
for every $1<q<s$.\\

We shall use the  next result.

\begin{lema}\cite[Lemma 2.4]{CU-M-P}\label{lema_equivalencias}
The following statements hold.
\begin{enumerate}
\item $w\in A_{\infty}$ if and only if $w=w_0w_1$, with $w_0\in A_1$ and $w_1\in \textrm{RH}_{\infty}$.
\item If $w\in A_1$  then $w^{-1}\in\textrm{RH}_{\infty}$.
\item If $u,v\in\textrm{RH}_{\infty}$ then $uv\in\textrm{RH}_{\infty}$.
\end{enumerate}
\end{lema}

A locally integrable function $f$ is of \emph{bounded mean
oscillation} if there exists a positive constant $C$ such that
\[\frac{1}{|Q|}\int_Q|f(x)-f_Q|\,dx \leq C\]
for every cube $Q\subset \R^n$, where $f_Q$ denotes the average
$|Q|^{-1}\int_Q f(y)\,dy$. In this case we write $f\in \bmo$,
and we consider the norm
\[\normbmo{f}:=\sup_{Q\subset \R^n} \frac{1}{|Q|}\int_Q|f-f_Q|\,dx.\]
In fact,  the function $\normbmo{\cdot}$ is not properly a norm since constant functions
have $\bmo$ norm equal to zero, but it is
 a norm on quotient space of $\bmo$
functions modulo the space of constant functions.
It is well known that every function $f\in \bmo$ satisfies the
John-Nierenberg inequality. More precisely, there exist two positive constants
$C_1$ and $C_2$, depending only on the dimension, such that for any
cube $Q$ in $\R^n$ and any $\lambda>0$ we have
\begin{equation}\label{eq: J-N}
|\{x\in Q: |f(x)-f_Q|>\lambda\}|\leq
C_1|Q|e^{-\frac{C_2\lambda}{\normbmo{f}}}.
\end{equation}
As a consequence of \eqref{eq: J-N} we obtain that for every
$1<p<\infty$, the quantity
\[\|f\|_{\textrm{BMO},p}:=\sup_{Q\subset \R^n}\left( \frac{1}{|Q|}\int_Q|f(x)-f_Q|^p\,dx\right)^{1/p}\]
is a norm on $\bmo$ equivalent to $\normbmo{\cdot}$ (see for example \cite{javi}).\\

We shall also consider the following  version of weighted $\bmo$ space.
Let $w$ be a weight. We say that a locally integrable function $f$
belongs to $\textrm{BMO}_w^*$ if
\[\normbmows{f}:=\sup_{Q\subset \R^n} \frac{1}{w(Q)}\int_Q |f(x)-f_Q|w(x)\,dx <\infty.\]
Note that $f_Q$ is defined as above, that is, $f_Q=|Q|^{-1}\int_Q
f(y)\,dy$, and $w(Q)=\int_Q w(x)\,dx$. We shall prove a relationship between $\bmo$  and
$\textrm{BMO}_w^*$ for  $w\in A_1$ in Lemma~\ref{lema1.1}.\\
%%%%%%%%%%%%%%%%%%%%%%%%%%%%%%%%%%%%%%%%%%%%%%%%%%%

We say that $\Phi:[0,\infty)\to [0,\infty]$ is a \emph{Young function} if it is strictly increasing, convex,  $\Phi(0)=0$ and $\Phi(t)\to\infty$ when $t\to\infty$. Given a Young function $\Phi$
 and a Muckenhoupt weight $w$, the \emph{generalized maximal
operator} $M_{\Phi,w}=M_{\Phi(L),w}$ is defined by
\[M_{\Phi,w} f(x):=\sup_{Q\ni x}\|f\|_{\Phi,Q,w},\]
where $\|f\|_{\Phi,Q,w}$ denotes the weighted $\Phi$-average over $Q$ defined by
means of the Luxemburg norm
\begin{equation}\label{eq: luxem norm}
\|f\|_{\Phi,Q,w}:=\inf\left\{\lambda>0:
\frac{1}{w(Q)}\int_Q\Phi\left(\frac{|f|}{\lambda}\right)w\,dx\leq
1\right\}.
\end{equation}
It can be proved that
\begin{equation}\label{eq: la norma lo hace}
\frac{1}{w(Q)}\int_Q\Phi\left(\frac{|f|}{\|f\|_{\Phi,Q,w}}\right)w\,dx\leq
1.
\end{equation}

By following the same arguments as in the result of Krasnosel'ski{\u\i} and Ruticki{\u\i} (\cite{KR}, see also \cite{raoren}), since  $d\mu(x)=w(x)\,dx$ is a doubling measure for $w\in A_\infty$, we can get that $\|f\|_{\Phi,Q,w}$ is equivalent to the following quantity

\[\inf_{\tau>0}\left\{\tau+\frac{\tau}{w(Q)}\int_{Q}\Phi\left(\frac{|f|}{\tau}\right)w\,dx\right\}.\]

\medskip

If $w=1$ we simply write $M_{\Phi}$ and $\|f\|_{\Phi,Q}$.
For example, when $\Phi(t)=t$, $M_\Phi$ is the Hardy-Littlewood
maximal operator $M$. The function $\Phi(t)=t(1+\log^{+}{t})^m$,
$m\in \N$, plays an important role in the estimations for
commutators of singular integrals. In this case, the corresponding
maximal function is denoted by  $M_{L(\log L)^m}$, which satisfies
\begin{equation}\label{eq: maximal equiv a M}
M_{L(\log L)^m}f(x)\approx M^{m+1}f(x),
\end{equation}
where $M^{m+1}$ denotes the composition of the maximal operator $m+1$ times
 with itself (see \cite{Carlos-95} and \cite{B-H-P}).\\

The next result is a well known fact about a relation between the $\normexpl{b-b_Q}$ and $\normbmo{b}$ when $b$ is a BMO function and its proof can be found in \cite{Carlos_endpoint}, where  $\normexpl{\cdot}$ denotes the $\Phi$-average over $Q$ when  $\Phi(t)=e^t-1$.

\begin{lema}\label{lema: perez}
Given  $f\in\bmo$, there exists a positive constant $C$ such that
\[\normexpl{f-f_Q}\leq C\normbmo{f}.\]
\end{lema}

\medskip

Given a Young function $\Phi$,
 we use $\bar{\Phi}$ to denote the
\emph{complementary Young function} associated to $\Phi$, defined
for $t\geq 0$ by
\[\bar{\Phi}(t)=\sup\{ts-\Phi(s):s\geq 0\}.\]
 It is well known that $\bar{\Phi}$ satisfies
\[t\leq \Phi^{-1}(t)\bar{\Phi}^{-1}(t)\leq 2t,\quad\forall t>0.\]

When $\Phi(t)=t(1+\log^{+}{t})^\alpha$, $\alpha>0$ we have that
$\bar{\Phi}(t)\approx\exp(t^{1/\alpha})-1$, with the corresponding
maximal function denoted by $M_{\exp L^{1/\alpha}}$. %, since
%\begin{equation}\label{normas equiv}
%\|f\|_{\bar{\phi},Q,w}\approx \|f\|_{\psi,Q,w},
%\end{equation}
%for every function $f$, every weight $w$ and every $Q$, where $\psi(t)=e^{1/\alpha}$ (justificar esto).

\medskip

The following
\emph{generalized H\"{o}lder inequality}
\begin{equation*}
\frac{1}{w(Q)}\int_Q |fg| w\,dx\leq 2 \|f\|_{\Phi,Q,w}\|g\|_{\bar{\Phi},Q,w}
\end{equation*}
holds.\\

%The following result will be a fundamental tool in the proof of our main result.

Given weights $u$ and $v$, by $v \in  A_p(u)$ we mean that $v$ satisfies the $A_p$ condition
with respect to the measure $\mu$ defined as $d\mu=u dx$. More precisely, for $1<p< \infty$, we say that
$v\in A_p(u)$ if there exists a positive constant $C$ such that
\[\left(\frac{1}{u(Q)}\int_Q v(x)u(x)\,dx \right)\left(\frac{1}{u(Q)}\int_Q v(x)^{-\frac{1}{p-1}}u(x)\,dx\right)^{p-1}\leq C,\]
for every cube $Q\subset\R^n$. A weight $v$ belongs to $A_1(u)$ if
\[\frac{1}{u(Q)}\int_Q v(x)u(x)\,dx\le C\, \inf_Q v(x).\]
We denote the union of all the $A_p(u)$ classes by $A_\infty(u)$.
We shall use the following result.

\begin{lema}\cite[Lemma 2.1]{CU-M-P}\label{lema1.6}
If $u\in A_1$ and $v\in A_{\infty}(u)$, then $uv\in A_{\infty}$. Particularly,
if  $v\in A_p(u)$ with $1\leq p<\infty$, then $uv\in
A_p$.
\end{lema}

Finally, we shall state a result concerning to a Coifman type inequality for commutators of Calder\'on-Zygmund operators, which  is proved in \cite{Carlos-sharp}.
\begin{lema}\label{lema: tipo fuerte Tbk}
Let $0<p<\infty$, $w\in A_{\infty}$ and $b\in \mathrm{BMO}$. Then there exists a positive constant $C$ such that
\[\int_{\R^n}|T_b^kf(x)|^pw(x)\,dx\leq C\normbmo{b}^{kp}[w]_{A_{\infty}}^{(k+1)p}\int_{\R^n} M^{k+1}f(x)^pw(x)\,dx.\]
\end{lema}
Note that when $w\in A_p$,  by applying $k + 1$ times Muckenhoupt's Theorem   we obtain the well known fact that the higher order commutators
are bounded on $L^p(w)$.

%
%A  \emph{Calder\'on-Zygmund singular integral} is a Calder\'on-Zygmund operator that satisfies
%\[Tf(x)=\lim_{\varepsilon\to 0} \int_{|x-y|>\varepsilon} K(x,y)f(y)\,dy.\]

\section{Auxiliary Lemmas}
In this section we prove four lemmas that we shall use in the proof of our results.
The first one states that the spaces $\bmo$ and
$\textrm{BMO}_w^*$  coincide when $w\in A_1$.

\begin{lema}\label{lema1.1}
Let $w\in A_1$. Then  $\normbmo{f}$ and $\normbmows{f}$ are
equivalent.
\end{lema}

\begin{proof}
Since $w$ belongs to  $A_1$, we have that
\begin{eqnarray*}
\frac{1}{|Q|}\int_Q |f(x)-f_Q|\,dx&= & \frac{1}{w(Q)}\frac{w(Q)}{|Q|}\int_Q |f(x)-f_Q|\,dx\\
&\leq&\frac{[w]_{A_1}}{w(Q)}\int_Q |f(x)-f_Q|w(x)\,dx\\
&\leq& [w]_{A_1}\normbmows{f}.
\end{eqnarray*}
On the other hand,   $w\in A_1$ implies that  there exists $s>1$
such that $w\in \textrm{RH}_s$. Then, from  H\"{o}lder's
inequality we obtain that
\begin{eqnarray*}
\frac{1}{w(Q)}\int_Q |f(x)-f_Q|w(x)\,dx&\leq & \frac{|Q|}{w(Q)}\left(\frac{1}{|Q|}\int_Q |f(x)-f_Q|^{s'}\right)^{1/s'}\left(\frac{1}{|Q|}\int_Q w^s\right)^{1/s}\\
&\leq&C[w]_{\textrm{RH}_s}\normbmo{f}\frac{|Q|}{w(Q)} \frac{1}{|Q|}\int_Q w(x)\,dx\\
&= & C[w]_{\textrm{RH}_s}\normbmo{f}.
\end{eqnarray*}
\end{proof}
%%%%%%%%%%%%%%%%%%%%%%%%%%%%%%%%%%%%%%%%%%%%%%%%%%%%%%

The next lemma gives us a way to deal with the weighted Orlicz norms, controlling them by the same non-weighted norms.

\begin{lema}\label{lema: con pesos menor que sin pesos}
Let $w$ be a weight such that $w\in\mathrm{RH}_s$ for some $s>1$. Then
\[\normexplw{f}\leq 2^{1/s'} [w]_{\mathrm{RH}_s}s'\normexpl{f}.\]
\end{lema}

\begin{proof}
Fix  $\lambda=s'\normexpl{f}$. In order to show that
$\normexplw{f}\leq 2^{1/s'}[w]_{\mathrm{RH}_s} \lambda$,
 it is enough to prove that
\begin{equation}\label{eq: con pesos menor que sin pesos}
\frac{1}{w(Q)}\int_Q \left(e^{\frac{|f(x)|}{\lambda}}-1\right)w(x)\,dx \leq 2^{1/s'}[w]_{\mathrm{RH}_s},
\end{equation}
for every cube $Q$. Indeed, since $C=2^{1/s'}[w]_{\mathrm{RH}_s}>1$, from \eqref{eq: con pesos menor que sin pesos}  we obtain that
\[\frac{1}{w(Q)}\int_Q \left(e^{\frac{|f(x)|}{C\lambda}}-1\right)w(x)\,dx\leq
\frac{1}{Cw(Q)}\int_Q \left(e^{\frac{|f(x)|}{\lambda}}-1\right)w(x)\,dx \leq 1,\]
where we have used that $\psi(\alpha t)\leq \alpha \psi(t)$, for every convex function $\psi$ with $\psi(0)=0$ and every $\alpha\in[0,1]$.
Then we conclude that $\normexplw{f}\leq C\lambda$.

Then, let us prove that \eqref{eq: con pesos menor que sin pesos} holds.
From  H\"{o}lder's inequality and the reverse H\"{o}lder condition $\mathrm{RH}_s$, we get that
\begin{eqnarray*}
\frac{1}{w(Q)}\int_Q \left(e^{\frac{|f(x)|}{\lambda}}-1\right)w(x)\,dx &\leq& \frac{1}{w(Q)}\int_Q e^{\frac{|f(x)|}{\lambda}}w(x)\,dx\\
&\leq&\frac{|Q|}{w(Q)}\left(\frac{1}{|Q|}\int_Q e^{\frac{|f|}{\normexpl{f}}}dx\right)^{1/s'}\left(\frac{1}{|Q|}\int_Q w^s\,dx\right)^{1/s}\\
&\leq&[w]_{\textrm{RH}_s}\left(\frac{1}{|Q|}\int_Q e^{\frac{|f|}{\normexpl{f}}}\,dx\right)^{1/s'}\\
&=&[w]_{\textrm{RH}_s}\left(\frac{1}{|Q|}\int_Q \left(e^{\frac{|f|}{\normexpl{f}}}-1\right)\,dx+1 \right)^{1/s'}\\
&\leq& 2^{1/s'}[w]_{\textrm{RH}_s},
\end{eqnarray*}
where in the last inequality we have used \eqref{eq: la norma lo hace}. We are done.
\end{proof}

The following result is useful in order to prove our main result.
\begin{lema}\label{lema:  promedios dilatados y BMO}
Given  $f\in\bmo$, there exists a positive constant $C$ such that
\[|f_Q-f_{2^kQ}|\leq C k\normbmo{f},\]
for every $k\in\N$ and every cube $Q$.
\end{lema}

\begin{proof}
Fix a cube $Q$ and a  positive integer $k$. Then
\begin{align*}
|f_Q-f_{2^kQ}|&\leq \sum_{j=0}^{k-1} |f_{2^{j+1}Q}-f_{2^jQ}|\\
&\leq \sum_{j=0}^{k-1} \frac{1}{|2^jQ|}\int_{2^jQ}|f(x)-f_{2^{j+1}Q}|\,dx\\
&\leq C\sum_{j=0}^{k-1} \frac{1}{|2^{j+1}Q|}\int_{2^{j+1}Q}|f(x)-f_{2^{j+1}Q}|\,dx\\
&\leq C k \normbmo{f}.
\end{align*}
\end{proof}

%\begin{lema}\cite[Thm. 5.1]{CU-Neu}  \label{lema_factorizacion}
%A weight  $w$ belongs to $\textrm{RH}_s\cap A_p$, with $1\leq p<\infty$ and
%$1<s\leq \infty$, if and only if there exist weights  $w_0,w_1$ such that
%$w_0\in A_1\cap \textrm{RH}_s$, $w_1\in A_p\cap\textrm{RH}_{\infty}$
%and $w=w_0w_1$.
%\end{lema}

For the proof of Theorem~\ref{teo: main} we shall use the following lemma which will be a fundamental tool. It
states that if $u\in A_1$ and $uv\in A_{\infty}$, then $u\in A_1(v)$.

\begin{lema}\label{lema: fundamental}
Let $u$ and $v$ be weights such that $u\in A_1$ y $uv\in
A_{\infty}$. Then there exists a positive constant $C$ such that
\begin{equation}\label{eq: fundamental}
\frac{(uv)(Q)}{v(Q)}\leq C\inf_Q u,
\end{equation}
for every cube $Q$.
\end{lema}

\begin{proof} Fix a cube $Q$. From (1) in Lemma~\ref{lema_equivalencias} there exist two weights $w_0,w_1$ such that $uv=w_0w_1$,
with  $w_0\in A_1$ and $w_1\in\textrm{RH}_{\infty}$. From the hypothesis on $u$ and (2),
$u^{-1}\in\textrm{RH}_{\infty}$, so that from (3) we can conclude that
 $w_1u^{-1}\in \textrm{RH}_{\infty}$. Let $s>1$ such that $w_0\in \textrm{RH}_s$.
Then, we have that
\begin{eqnarray*}
\frac{(uv)(Q)}{v(Q)}&=&\frac{(w_0w_1)(Q)}{(w_0w_1u^{-1})(Q)}\\
&\leq&\frac{\int_Q w_0w_1\,dx}{(\inf_Q w_0)\int_Qw_1u^{-1}\,dx}\\
&\leq&[w_1u^{-1}]_{\textrm{RH}_{\infty}}\frac{1}{\inf_Q w_0}\frac{1}{\sup_Q (w_1u^{-1})}\frac{1}{|Q|}\int_Q w_0w_1\,dx\\
&\leq &[w_1u^{-1}]_{\textrm{RH}_{\infty}}\frac{1}{\inf_Q w_0}\frac{1}{ \sup_Q (w_1u^{-1})}\left(\frac{1}{|Q|}\int_Q w_0^s\right)^{1/s}\left(\frac{1}{|Q|}\int_Q w_1^{s'}\right)^{1/s'}\\
&\leq &[w_1u^{-1}]_{\textrm{RH}_{\infty}}\frac{1}{\inf_Q w_0}\frac{1}{\sup_Q (w_1u^{-1})}\frac{[w_0]_{\textrm{RH}_s}w_0(Q)}{|Q|}\frac{[w_1]_{\textrm{RH}_{s'}}w_1(Q)}{|Q|}\\
&\leq&[w_0]_{A_1}[w_0]_{\textrm{RH}_s}[w_1]_{\textrm{RH}_{s'}}[w_1u^{-1}]_{\textrm{RH}_{\infty}}\frac{1}{|Q|}\int_Q \frac{w_1u^{-1}u\,dx}{\sup_Q{(w_1u^{-1})}}\\
&\leq&[w_0]_{A_1}[w_0]_{\textrm{RH}_s}[w_1]_{\textrm{RH}_{s'}}[w_1u^{-1}]_{\textrm{RH}_{\infty}}\frac{u(Q)}{|Q|}\\
&\leq&[w_0]_{A_1}[w_0]_{\textrm{RH}_s}[w_1]_{\textrm{RH}_{s'}}[w_1u^{-1}]_{\textrm{RH}_{\infty}}[u]_{A_1}\inf_Q
u.
\end{eqnarray*}
\end{proof}

%Given a locally integrable function $b$, the \emph{commutator operator} $[b, T]$ is defined for  adequate functions $f$ by
%\[[b, T]f=bT(f)-T(bf).\]

\section{Proof of the main results}

We shall use the following result about  Calder\'on-Zygmund operators.

\begin{teo}\cite[Thm. 1.3]{CU-M-P}\label{teo1.7}
If $u,v$ are weights such that $u,v\in A_1$, or $u\in A_1$ and $v\in
A_{\infty}(u)$, then there exists a positive constant $C$ such that for every
$t>0$,
%\[uv\left\{x\in\R^n: \frac{\mathcal{M}(fv)(x)}{v(x)}>t\right\}\leq \frac{C}{t}\int_{\R^n}|f(x)|u(x)v(x)\,dx,\]
%and
\[uv\left(\left\{x\in\R^n: \frac{T(fv)(x)}{v(x)}>t\right\}\right)\leq \frac{C}{t}\int_{\R^n}|f(x)|u(x)v(x)\,dx,\]
where %$\mathcal{M}$ Hardy-Littlewood  maximal operator and
 $T$ is any Calder\'on-Zygmund
operator.
\end{teo}

\begin{proof}[Proof of Theorem~\ref{teo: main}] Note that since $[b,T](f/\normbmo{b})=[b/\normbmo{b}, T] f$, we can assume that  $\normbmo{b}=1$. Without loss of generality, we can assume that $f$ is a bounded, non-negative function with compact support.
Fix $t>0$, and form the  Calder\'on-Zygmund decomposition of $f$ at height $t>0$ with respect to the doubling measure $\mu$ given by
$d\mu(x)=v(x)\,dx$ ($\mu$ is doubling since $v\in A_{\infty}(u)$ implies $v\in A_{\infty}$ (see \cite[Lemma 2.1]{CU-M-P})).
This yields a collection of disjoint dyadic cubes $\{Q_j\}_{j=1}^{\infty}$, such that
$t<f_{Q_j}^v\leq Ct$ for some $C>1$, where $f_{Q_j}^v$ is defined by
\[f_{Q_j}^v=\frac{1}{v(Q_j)}\int_{Q_j}f(y)v(y)\,dy.\]
From this decomposition we have that if $\Omega=\bigcup_{j=1}^{\infty}Q_j$ then $f(x)\leq t$ in almost every $x\in{\R}^n\backslash\Omega$.

We decompose $f$ as $f=g+h$, where
\begin{equation*}
g(x)=\left\{
\begin{array}{ccl}
f(x),& \textrm{ if } &x\in\R^{n}\backslash\Omega;\\
f_{Q_j}^v,&\textrm{ if }& x\in Q_j,
\end{array}
\right.
\end{equation*}
and $h(x)=\s{j=0}{\infty}{h_j(x)}$, with
\begin{equation*}
h_j(x)=\left(f(x)-f_{Q_j}^v\right)\mathcal{X}_{Q_j}(x),
\end{equation*}
where $\mathcal X_E$ denotes the indicator function in the set $E$.
It follows that $g(x)\leq Ct$ almost everywhere, each
 $h_j$ is supported on $Q_j$ and
\begin{equation}\label{integral_cero_v}
\int_{Q_j}h_j(y)v(y)\,dy=0.
\end{equation}
With $cQ$, $c>0$, we will denote the cube concentric with $Q$ whose side length is $c$ times the side length of $Q$. So let $Q_j^*=3Q_j$ and $\Omega^*=\bigcup_j Q_j^*$. Then
\begin{align*}
uv\left(\left\{x\in\R^n: \left|\frac{[b,T](fv)(x)}{v(x)}\right|>t\right\}\right)& \leq uv\left(\left\{x\in\R^n: \left|\frac{[b,T](gv)(x)}{v(x)}\right|>\frac{t}{2}\right\}\right)\\
&+uv\left(\left\{x\in\R^n: \left|\frac{[b,T](hv)(x)}{v(x)}\right|>\frac{t}{2}\right\}\right)\\
&\leq uv\left(\left\{x\in\R^n: \left|\frac{[b,T](gv)(x)}{v(x)}\right|>\frac{t}{2}\right\}\right)\\&+(uv)(\Omega^*)\\
&+uv\left(\left\{x\in\R^n\backslash\Omega^*: \left|\frac{[b,T](hv)(x)}{v(x)}\right|>\frac{t}{2}\right\}\right)\\
&=I+II+III.
\end{align*}
We shall estimate each term separately.  Since $v\in
A_{\infty}(u)$, there exists $q'>1$ such that $v\in A_{q'}(u)$, so that
$v^{1-q}\in A_q(u)$ and  by Lemma~\ref{lema1.6} we have that
$uv^{1-q}\in A_q$. By applying Tchebychev's inequality with
 $q>1$  we obtain
\begin{align*}
I&=uv\left(\left\{x\in\R^n: \left|\frac{[b,T](gv)(x)}{v(x)}\right|>\frac{t}{2}\right\}\right)\\
&\leq \frac{C}{t^q}\int_{\R^n} |[b,T](gv)(x)|^qu(x)v^{1-q}(x)\,dx\\
&\leq \frac{C}{t^q}\int_{\R^n} g(x)^qu(x)v(x)\,dx\\
&\leq \frac{C}{t}\int_{\R^n} g(x)u(x)v(x)\,dx,
\end{align*}
since $uv^{1-q}\in A_{q}$ implies that the commutator  $[b,T]$ is bounded
on $L^q(uv^{1-q})$ (see remark after Lemma~\ref{lema: tipo fuerte Tbk}) and   $g(x)\leq Ct$.
From the definition of $g$ and Lemma~\ref{lema: fundamental} we get that
\begin{eqnarray*}
I&\leq& \frac{C}{t}\int_{\R^n\backslash\Omega}
f(x)u(x)v(x)\,dx+\frac{C}{t}\s{j=1}{\infty}{\frac{(uv)(Q_j)}{v(Q_j)}\int_{Q_j}f(y)v(y)\,dy}\\
&\leq& \frac{C}{t}\int_{\R^n\backslash\Omega}f(x)u(x)v(x)\,dx+\frac{C}{t}\s{j=1}{\infty}{\int_{Q_j}f(x)u(x)v(x)\,dx}\\
&\leq&\frac{C}{t}\int_{\R^n}f(x)u(x)v(x)\,dx.
\end{eqnarray*}
In order to estimate $II$, since $uv$ is doubling and by applying Lemma~\ref{lema: fundamental} we get
\begin{eqnarray*}
II&=&(uv)(\Omega^*)=\s{j}{}{(uv)(Q_j^*)}\\
&\leq&C\s{j}{}{v(Q_j)\frac{(uv)(Q_j)}{v(Q_j)}}\\
&\leq&C\s{j}{}{(\inf_{Q_j}u)\frac{1}{t}\int_{Q_j}f(x)v(x)\,dx}\\
&\leq&\frac{C}{t}\int_{\R^n}f(x)u(x)v(x)\,dx.
\end{eqnarray*}
By observing that
\[\frac{[b,T](hv)}{v}=\s{j}{}{\frac{[b,T](h_jv)}{v}}=\s{j}{}{\frac{(b-b_{Q_j})T(h_jv)}{v}}-\s{j}{}{\frac{T\left((b-b_{Q_j})h_jv\right)}{v}},\]
the estimate of $III$ can be made as follows
\begin{eqnarray*}
III&\leq&uv\left(\left\{x\in\R^n\backslash\Omega^*: \left|\s{j}{}{\frac{(b-b_{Q_j})T(h_jv)}{v}}\right|>\frac{t}{4}\right\}\right)\\
&&+uv\left(\left\{x\in\R^n\backslash\Omega^*: \left|\s{j}{}{\frac{T((b-b_{Q_j})h_jv)}{v}}\right|>\frac{t}{4}\right\}\right)\\
=A+B.
\end{eqnarray*}
Let us first estimate $A$. From the Tchebychev's inequality, Tonelli's theorem and \eqref{integral_cero_v} we have that
\begin{eqnarray*}
A&\leq&\frac{C}{t}\int_{\R^n\backslash\Omega^*}\s{j}{}{|b(x)-b_{Q_j}||T(h_jv)(x)|u(x)\,dx}\\
&\leq&\frac{C}{t}\s{j}{}{\int_{\R^n\backslash Q_j^*}|b(x)-b_{Q_j}|\left|\int_{Q_j}h_j(y)v(y)K(x-y)\,dy\right|u(x)\,dx}\\
&=&\frac{C}{t}\s{j}{}{\int_{\R^n\backslash Q_j^*}|b(x)-b_{Q_j}|\left|\int_{Q_j}h_j(y)v(y)\left[K(x-y)-K(x-x_{Q_j})\right]\,dy\right|u(x)\,dx}\\
&\leq&\frac{C}{t}\s{j}{}{\int_{Q_j}|h_j(y)|v(y)\int_{\R^n\backslash
Q_j^*}\left|b(x)-b_{Q_j}\right|\left|K(x-y)-K(x-x_{Q_j})\right|u(x)\,dx\,dy}.
\end{eqnarray*}
Given a cube $Q_j$, denote $x_{Q_j}$
its center, $\ell(Q_j)$ the length of its side, $r_j=2^{-1}\ell(Q_j)$ and
$A_{j,k}=\{x: 2^kr_j\leq |x-x_{Q_j}|<2^{k+1}r_j\}$. Then, for each $y\in Q_j$, from \eqref{eq:prop del nucleo} we have
\begin{align*}
\int_{\R^n\backslash Q_j^*}|b(x)-b_{Q_j}||K(x-y)&-K(x-x_{Q_j})|u(x)\,dx\\ &\leq\s{k=1}{\infty}{\int_{A_{j,k}}|b(x)-b_{Q_j}|\frac{|y-x_{Q_j}|}{|x-x_{Q_j}|^{n+1}}u(x)\,dx}\\
&\leq\s{k=1}{\infty}{\frac{r_j}{(2^kr_j)^{n+1}}\int_{2^{k+1}Q_j}|b(x)-b_{Q_j}|u(x)\,dx}.
\end{align*}
Then, from Lemmas~\ref{lema1.1} and~\ref{lema:  promedios dilatados y BMO}, and the fact that $u\in A_1$,
\begin{align*}
\s{k=1}{\infty}{\frac{r_j}{(2^kr_j)^{n+1}}\int_{2^{k+1}Q_j}|b-b_{Q_j}|u\,dx}&\leq
 C\s{k=1}{\infty}{\frac{2^{-k}}{|2^{k+1}Q_j|}\int_{2^{k+1}Q_j}|b-b_{Q_j}|u\,dx}\\
&\leq C\s{k=1}{\infty}{\frac{2^{-k}}{|2^{k+1}Q_j|}\int_{2^{k+1}Q_j}|b-b_{2^{k+1}Q_j}|u\,dx}\\
&+ C\s{k=1}{\infty}{\frac{2^{-k}}{|2^{k+1}Q_j|}\int_{2^{k+1}Q_j}|b_{2^{k+1}Q_j}-b_{Q_j}|u\,dx}\\
&\leq C\s{k=1}{\infty}{2^{-k}\frac{u(2^{k+1}Q_j)}{|2^{k+1}Q_j|}\frac{1}{u(2^{k+1}Q_j)}\int_{2^{k+1}Q_j}|b-b_{2^{k+1}Q_j}|u\,dx}\\
&+C\s{k=1}{\infty}{2^{-k}C(k+1)u(y)}\\
&\leq  Cu(y),
\end{align*}
 Hence, from Lemma~\ref{lema: fundamental} we obtain that
\begin{eqnarray*}
A&\leq&\frac{C}{t}\s{j}{}{\int_{Q_j}|h_j(y)|u(y)v(y)\,dy}\\
&\leq&\frac{C}{t}\s{j}{}{\left(\int_{Q_j}f(y)u(y)v(y)\,dy+\int_{Q_j}f_{Q_j}^{v}u(y)v(y)\,dy\right)}\\
&\leq&\frac{C}{t}\s{j}{}{\left(\int_{Q_j}f(y)u(y)v(y)\,dy+\frac{(uv)(Q_j)}{v(Q_j)}\int_{Q_j}f(y)v(y)\,dy\right)}\\
&\leq&\frac{C}{t}\int_{\R^n}f(y)u(y)v(y)\,dy,
\end{eqnarray*}
 We shall finally estimate
 $B$. By applying Theorem~\ref{teo1.7} we get
\begin{eqnarray*}
B&\leq&\frac{C}{t}\int_{\R^n}\left|\s{j}{}{(b(x)-b_{Q_j})h_j(x)}\right|u(x)v(x)\,dx\\
&\leq&\frac{C}{t}\s{j}{}{\int_{Q_j}|b(x)-b_{Q_j}|f(x)u(x)v(x)\,dx}+\frac{C}{t}\s{j}{}{\int_{Q_j}|b(x)-b_{Q_j}|f_{Q_j}^{v}u(x)v(x)\,dx}\\
&=&B_1+B_2.
\end{eqnarray*}
From the generalized H\"{o}lder's inequality with respect to the doubling measure
$w=uv$, Lemma~\ref{lema: con pesos menor que sin pesos} and Lemmas~\ref{lema: perez}
and~\ref{lema: fundamental},
we obtain
\begin{eqnarray*}
B_1&\leq&\frac{C}{t}\s{j}{}{(uv)(Q_j)\normexplv{b-b_{Q_j}}{uv}{Q_j}\normlloglv{f}{uv}{Q_j}}\\
&\leq&\frac{C}{t}\s{j}{}{(uv)(Q_j)\|b-b_{Q_j}\|_{\textrm{exp}L,Q_j}\inf_{\tau>0}\left\{\tau+\frac{\tau}{(uv)(Q_j)}\int_{Q_j}\Phi\left(\frac{f}{\tau}\right)uv\,dx\right\}}\\
&\leq & C\s{j}{}{\left((uv)(Q_j)+\int_{Q_j}\Phi\left(\frac{f}{t}\right)uv\,dx\right)}\\
&\leq & C\s{j}{}{\left(v(Q_j)\inf_{Q_j}u+\int_{Q_j}\Phi\left(\frac{f}{t}\right)uv\,dx\right)}\\
&\leq & C\s{j}{}{\left(\inf_{Q_j}u\frac{1}{t}\int_{Q_j}fv\,dx+\int_{Q_j}\Phi\left(\frac{f}{t}\right)uv\,dx\right)}\\
&\leq&C\int_{\R^n}\Phi\left(\frac{f}{t}\right)uv\,dx.
\end{eqnarray*}
In order to estimate $B_2$, let $s$ be the reverse H\"{o}lder exponent of the weight $uv$. Then, by applying H\"{o}lder's inequality we get
\begin{eqnarray*}
B_2&\leq&\frac{C}{t}\s{j}{}{\frac{|Q_j|}{v(Q_j)}\left(\int_{Q_j}f(y)v(y)\,dy\right)\left(\frac{1}{|Q_j|}\int_{Q_j}|b-b_{Q_j}|^{s'}\right)^{1/s'}\hspace*{-0.2cm}\left(\frac{1}{|Q_j|}\int_{Q_j}(uv)^s\right)^{1/s}}\\
&\leq&\frac{C}{t}\s{j}{}{\frac{(uv)(Q_j)}{v(Q_j)}\int_{Q_j}f(y)v(y)\,dy}\\
&\leq&C\frac{1}{t}\s{j}{}{\inf_{Q_j}u\int_{Q_j}f(y)v(y)\,dy}\\
&\leq&C\frac{1}{t}\int_{\R^n}f(y)u(y)v(y)\,dy,
\end{eqnarray*}
and the result is proved.
\end{proof}

%\section{Higher-order commutators}

%In order to show the theorem above, we shall state and prove the following technical lemma.

%\begin{lema}\label{lema: factorizacion con inversas} Let $A,B,C$ be continuous and  a functions. Define $A^{-1}(x)=\inf\{y: A(y)>x\}$, where $\inf \emptyset =\infty$.
%If for every $0\leq x< \infty$ we have
%\[A^{-1}(x)B^{-1}(x)\leq C^{-1}(x),\]
%then $C(xy)\leq A(x)+B(y)$,  for every $0 \leq x,y<\infty$.
%\end{lema}
%
%\begin{proof}
%We claim that $A(A^{-1}(x))\leq x\leq A^{-1}(A(x))$. Indeed,
%\[A(A^{-1}(x))=A(\inf\{y: A(y)>x\})=\inf\{A(y): A(y)>x\}=x,\]
%since $A$ is a continuous function. On the other hand, if $y$ is such that $A(y)>A(x)$,
%being  $A$ increasing we have that $y>x$. Then $x<y$ for every $y$ that satisfies $A(y)>A(x)$.
%Hence
%\[x\leq \inf\{y: A(y)>A(x)\}=A^{-1}(A(x)).\]
%Let us note also that if $x\leq z$ then $A^{-1}(x)\leq A^{-1}(z)$, since
%\[\{y: A(y)>z\}\subseteq \{y: A(y)>x\}.\]
%We shall use that the two facts above hold for $A$, $B$ and $C$.
%
%Fix  $0\leq x,y<\infty$. Assume first that $A(x)\leq B(y)$. Then
%\[xy\leq A^{-1}(A(x))B^{-1}(B(y))\leq A^{-1}(B(y))B^{-1}(B(y))\leq C^{-1}(B(y)),\]
%and we have that $C(xy)\leq C(C^{-1}(B(y)))\leq B(y)$.
%On the other hand, if  $A(x)>B(y)$ we have
%\[xy\leq A^{-1}(A(x))B^{-1}(B(y))\leq A^{-1}(A(x))B^{-1}(A(x))\leq C^{-1}(A(x)),\]
%and then $C(xy)\leq C(C^{-1}(A(x)))\leq A(x)$.
%%
%Hence, \[C(xy)\leq \max\{A(x),B(y)\}\leq A(x)+B(y),\] which proves the result.
%\end{proof}

\begin{proof}[Proof of Theorem~\ref{teo: orden m}]
Again, since $T^m_b(f/\normbmo{b})=T^m_{b/\normbmo{b}}(f)$, we assume that $\normbmo{b}=1$.
Without loss of generality, we can also assume that $f$ is a bounded, non-negative function with compact support.
Fix a positive integer $m$. We will use an induction argument.
The case $m=1$ is proved in Theorem~\ref{teo: main}.
Assume that the result holds for every
 $1\leq k\leq m-1$.
For a fixed $t>0$, we consider again the
 Calder\'on-Zygmund decomposition of  $f$ at height $t$ with respect to the doubling measure $\mu$ given by
$d\mu(x)=v(x)\,dx$. Then, with the same notation as in the proof of Theorem~\ref{teo: main}, we have that
\begin{eqnarray*}
uv\left(\left\{x\in \R^n: \left|\frac{T_b^m(fv)(x)}{v(x)}\right|>t\right\}\right)&\leq&uv\left(\left\{x\in \R^n: \left|\frac{T_b^m(gv)(x)}{v(x)}\right|>\frac{t}{2}\right\}\right)\\
&&+uv(\Omega^*)+uv\left(\left\{x\in \R^n\backslash\Omega^*: \left|\frac{T_b^m(hv)(x)}{v(x)}\right|>\frac{t}{2}\right\}\right)\\
&=&I+II+III.
\end{eqnarray*}

We shall first estimate $I$. Under the hypothesis on $u$ and $v$, there exists $q>1$ such that $uv^{1-q}\in A_q$.
From Tchebychev's inequality and the remark after Lemma~\ref{lema: tipo fuerte Tbk} we obtain
\begin{eqnarray*}I&\leq&\frac{C}{t^q}\int_{\R^n}|T_b^m(gv)(x)|^qu(x)v^{1-q}(x)\,dx\\
&\leq&\frac{C}{t^q}\int_{\R^n}|g(x)|^qu(x)v(x)\,dx\\
&\leq&\frac{C}{t}\int_{\R^n}|g(x)|u(x)v(x)\,dx,
\end{eqnarray*}
and now we proceed as in the proof of Theorem~\ref{teo: main} to obtain the desired estimate. The estimation of $II$  is obtained exactly as in that theorem.

Then, we shall focus on $III$.
Observe that $h_jv$ is supported on $Q_j$, so that if $x\not\in Q_j$
\begin{align*}
T_b^m(h_jv)(x)&= \int_{\R^n}\left(b(x)-b(y)\right)^mK(x-y)h_j(y)v(y)\,dy\\
&= \s{\ell=0}{m}{C_{\ell,m}(b(x)-b_{Q_j})^{m-\ell}\int_{\R^n}(b(y)-b_{Q_j})^\ell K(x-y)h_j(y)v(y)\,dy}\\     %(-1)^\ell
&= C_{0,m}(b(x)-b_{Q_j})^mT(h_jv)(x)+C_{m,m}T\left((b-b_{Q_j})^mh_jv\right)(x)\\    %(-1)^m
& +\s{\ell=1}{m-1}{C_{\ell,m}(b(x)-b_{Q_j})^{m-\ell}\int_{\R^n}(b(y)-b_{Q_j})^\ell K(x-y)h_j(y)v(y)\,dy}.  %(-1)^\ell
\end{align*}
Note that, expanding as before the binomial expression $(b(x)-b_{Q_j})^{m-\ell}$, we obtain
\begin{align*}
\s{\ell=1}{m-1}{}&C_{\ell,m}(b(x)-b_{Q_j})^{m-\ell}\int_{\R^n}(b(y)-b_{Q_j})^\ell K(x-y) h_j(y)v(y)\,dy\\  %(-1)^\ell
&=\s{\ell=1}{m-1}{}C_{\ell,m}\s{i=0}{m-\ell}{C_{i,\ell,m}}\int_{\R^n}(b(x)-b(y))^i(b(y)-b_{Q_j})^{m-i}K(x-y)h_j(y)v(y)\,dy\\
&=\s{i=1}{m-1}{}\s{\ell=1}{m-i}C_{\ell,m}C_{i,m,\ell}\int_{\R^n}(b(x)-b(y))^i(b(y)-b_{Q_j})^{m-i}K(x-y)h_j(y)v(y)\,dy\\
&+\s{\ell=1,i=0}{m-1}{}C_{\ell,m}C_{0,m,\ell}\int_{\R^n}(b(y)-b_{Q_j})^mK(x-y)h_j(y)v(y)\,dy.
\end{align*}
%(see Figure~1).
%%%%%%%%%%%%%%%%%%%%%%%%%%%%%%%%%%%
%\vspace*{-0.5cm}
%\begin{figure}[h!tb]\label{fig: indices suma}
%\begin{center}
%\begin{tikzpicture}[scale=0.75]
%\draw[-stealth] (-0.2,0) -- (5,0);
%\draw[-stealth] (0,-0.2) -- (0,4.5);
%\foreach \x in {1,...,8}
%\draw[fill=black] (\x/2,0) circle (1.5 pt);
%\foreach \x in {1,...,8}
%\draw[|-](\x/2,0)--(\x/2+0.5,0);
%\foreach \x in {1,...,8}
%\draw[|-](0,\x/2)--(0,\x/2+0.5);
%\foreach \x in {1,...,7}
%\draw[fill=black] (\x/2,0.5) circle (1.5 pt);
%\foreach \x in {1,...,6}
%\draw[fill=black] (\x/2,1) circle (1.5 pt);
%\foreach \x in {1,...,5}
%\draw[fill=black] (\x/2,1.5) circle (1.5 pt);
%\foreach \x in {1,...,4}
%\draw[fill=black] (\x/2,2) circle (1.5 pt);
%\foreach \x in {1,...,3}
%\draw[fill=black] (\x/2,2.5) circle (1.5 pt);
%\foreach \x in {1,...,2}
%\draw[fill=black] (\x/2,3) circle (1.5 pt);
%\draw[fill=black] (1/2,3.5) circle (1.5 pt);
%\draw[dashed,rounded corners] (0.3,-0.2)--(4.5,-0.2)--(0.3,4)--cycle;
%\node[left] at (0,4) {$m$};
%\node[left] at (0,0.5) {$1$};
%\node[left] at (0,4.5) {$i$};
%\node[below] at (0.5,-0.2) {$1$};
%\node[below] at (4,-0.2) {$m-1$};
%\node[below] at (5,-0.2) {$\ell$};
%\end{tikzpicture}
%\caption{$1\leq\ell\leq m-1$; $0\leq i\leq m-\ell$}
%\end{center}
%\end{figure}
%%%%%%%%%%%%%%%%%%%%%%%%%%%%%%%%%%%%%

\noindent Thus we can write
\begin{align*}
\s{i=0}{m-1}{}&C_{i,m}\int_{\R^n} (b(x)-b(y))^i(b(y)-b_{Q_j})^{m-i}K(x-y)h_j(y)v(y)\,dy\\
& =C_{0,m}T((b-b_{Q_j})^mh_jv)(x)+\s{i=1}{m-1}{C_{i,m}T_b^i\left((b-b_{Q_j})^{m-i}h_jv\right)(x)},
\end{align*}
so that
\begin{align*}
\left|\s{j}{}{T_b^m(h_jv)(x)}\right|&=\left|C_{0,m}\s{j}{}{(b(x)-b_{Q_j})^mT(h_jv)(x)}\right|\\
&+\left|C_{m,m}\s{j}{}{T\left((b-b_{Q_j})^mh_jv\right)(x)}\right|\\
&+ \left|\s{j}{}{}\s{i=1}{m-1}{C_{i,m}T_b^i((b-b_{Q_j})^{m-i}h_jv)(x)}\right|.
\end{align*}
Then we can estimate $III$ as follows:
\begin{align*}
III&\leq uv\left(\left\{x\in\R^n\backslash\Omega^* : \left|C_{0,m}\s{j}{}{\frac{(b-b_{Q_j})^mT(h_jv)(x)}{v(x)}}\right|>\frac{t}{6}\right\}\right)\\
&+uv\left(\left\{x\in\R^n\backslash\Omega^* : \left|C_{m,m}\s{j}{}{\frac{T((b-b_{Q_j})^mh_jv)(x)}{v(x)}}\right|>\frac{t}{6}\right\}\right)\\
&+uv\left(\left\{x\in\R^n\backslash\Omega^* : \left|\s{i=1}{m-1}{C_{i,m}\frac{T_b^i(\s{j}{}{}(b-b_{Q_j})^{m-i}h_jv)(x)}{v(x)}}\right|>\frac{t}{6}\right\}\right)\\
&= I_1+I_2+I_3.
\end{align*}
In order to estimate $I_1$, we use Tchebychev's inequality  to obtain
\begin{align*}
I_1&\leq\frac{C}{t}\int_{\R^n\backslash\Omega^*}\left|\s{j}{}{(b(x)-b_{Q_j})^mT(h_jv)(x)}\right|u(x)\,dx\\
&\leq\frac{C}{t}\s{j}{}{\int_{\R^n\backslash Q_j^*}|b(x)-b_{Q_j}|^m\left|\int_{Q_j}\hspace*{-0.2cm}\left(K(x-y)-K(x-x_{Q_j})\right)h_j(y)v(y)\,dy\right|u(x)\,dx}\\
&\leq\frac{C}{t}\s{j}{}{\int_{Q_j}|h_j(y)|v(y)\int_{\R^n\backslash Q_j^*}|b(x)-b_{Q_j}|^m|K(x-y)-K(x-x_{Q_j})|u(x)\,dx\,dy}\\
&\leq\frac{C}{t}\s{j}{}{\int_{Q_j}|h_j(y)|v(y)\s{k=1}{\infty}{\int_{A_{j,k}}|b(x)-b_{Q_j}|^m\frac{|y-x_{Q_j}|}{|x-x_{Q_j}|^{n+1}}u(x)\,dx\,dy}}\\
&\leq\frac{C}{t}\s{j}{}{\int_{Q_j}|h_j(y)|v(y)\s{k=1}{\infty}{\frac{2^{-k}}{|2^{k+1}Q_j|}\int_{2^{k+1}Q_j}|b(x)-b_{Q_j}|^mu(x)\,dx\,dy}},
\end{align*}
where $A_{j,k}$ is the set defined in the proof of Theorem~\ref{teo: main}.
For $y\in Q_j$ we can bound the sum over $k$ by
\begin{align*}
\s{k=1}{\infty}{}\frac{2^{-k}}{|2^{k+1}Q_j|}&\int_{2^{k+1}Q_j}|b(x)-b_{Q_j}|^mu(x)\,dx\\
&\leq2^m\s{k=1}{\infty}{\frac{2^{-k}}{|2^{k+1}Q_j|}\int_{2^{k+1}Q_j}|b(x)-b_{2^{k+1}Q_j}|^mu(x)\,dx}\\
&+2^m\s{k=1}{\infty}{\frac{2^{-k}}{|2^{k+1}Q_j|}\int_{2^{k+1}Q_j}|b_{Q_j}-b_{2^{k+1}Q_j}|^mu(x)\,dx}.
\end{align*}

With a change of variables we easily get that
\begin{equation}\label{eq: expL1/m,expL}
\|g^m\|_{\mathrm{exp}L^{1/m},Q}=\|g\|^m_{\mathrm{exp}L,Q},
\end{equation}
for every cube $Q$. From this fact and the generalized
H\"{o}lder's inequality we obtain
\begin{align*}
2^m\s{k=1}{\infty}{}\frac{2^{-k}}{|2^{k+1}Q_j|}&\int_{2^{k+1}Q_j}|b(x)-b_{2^{k+1}Q_j}|^mu(x)\,dx\\
&\leq C\|(b-b_{2^{k+1}Q_j})^m\|_{\mathrm{exp}L^{1/m},2^{k+1}Q_j}\|u\|_{L(\mathrm{log}L)^m,2^{k+1}Q_j}\\&\leq
C\|b-b_{2^{k+1}Q_j}\|_{\mathrm{exp}L,2^{k+1}Q_j}^mM_{L(\mathrm{log}L)^m}u(y)\\&\leq CM^{m+1}u(y)\\&\leq Cu(y),
\end{align*}
where we have used \eqref{eq: maximal equiv a M} and that $u\in A_1$.

On the other hand, from Lemma~\ref{lema:  promedios dilatados y BMO} we have that
\begin{align*}
2^m\s{k=1}{\infty}{}\frac{2^{-k}}{|2^{k+1}Q_j|}&\int_{2^{k+1}Q_j}|b_{Q_j}-b_{2^{k+1}Q_j}|^mu(x)\,dx\\
&\leq C\s{k=1}{\infty}{\frac{2^{-k}}{|2^{k+1}Q_j|}\int_{2^{k+1}Q_j}(k+1)^m u(x)\,dx}\\
&\leq Cu(y)\s{k=1}{\infty}{2^{-k}(k+1)^m}\\
&=Cu(y).
\end{align*}
Hence
\[I_1\leq\frac{C}{t}\s{j}{}{\int_{Q_j}|h_j(y)|u(y)v(y)\,dy},\]
and then we proceed as in the proof of Theorem~\ref{teo: main}.
In order to estimate $I_2$, we use Theorem~\ref{teo1.7} to get that
\begin{align*}
I_2&\leq\frac{C}{t}\int_{\R^n}\left|\s{j}{}{(b(x)-b_{Q_j})^mh_j(x)}\right|u(x)v(x)\,dx\\
&\leq\frac{C}{t}\s{j}{}{\int_{Q_j}|b(x)-b_{Q_j}|^m|h_j(x)|u(x)v(x)\,dx}\\
&\leq\frac{C}{t}\s{j}{}{\int_{Q_j}|b(x)-b_{Q_j}|^mf(x)u(x)v(x)\,dx}\\
&+\frac{C}{t}\s{j}{}{\int_{Q_j}|b(x)-b_{Q_j}|^mf_{Q_j}^v u(x)v(x)\,dx}\\
&=I_{2,1}+I_{2,2}.
\end{align*}
From the generalized H\"{o}lder's inequality with respect to the measure $\mu$ defined by $d\mu=uv\,dx$, and by applying \eqref{eq: expL1/m,expL} we have that
\begin{align*}
I_{2,1}&\leq\frac{C}{t}\s{j}{}{(uv)(Q_j)\|(b-b_{Q_j})^m\|_{\mathrm{exp}L^{1/m},Q_j,uv}\|f\|_{L(\mathrm{log}L)^m,Q_j,uv}}\\
&\leq\frac{C}{t}\s{j}{}{(uv)(Q_j)}\|b-b_{Q_j}\|_{\mathrm{exp}L,Q_j,uv}^m\left\{t+\frac{t}{uv(Q_j)}\int_{Q_j}\Phi_m\left(\frac{f}{t}\right)uv\,dx\right\}\\
&\leq C\s{j}{}{\left(\frac{(uv)(Q_j)}{v(Q_j)}v(Q_j)+\int_{Q_j}\Phi_m\left(\frac{f(x)}{t}\right)u(x)v(x)\,dx \right)}\\
&\leq C\s{j}{}{\left(\frac{1}{t}\int_{Q_j}f(x)u(x)v(x)\,dx+\int_{Q_j}\Phi_m\left(\frac{f(x)}{t}\right)u(x)v(x)\,dx\right)}\\
&\leq C\int_{\R^n}\Phi_m\left(\frac{f(x)}{t}\right)u(x)v(x)\,dx,
\end{align*}
where we have used Lemma~\ref{lema: fundamental}.

For  $I_{2,2}$, we apply H\"{o}lder's inequality with exponent $s$ y $s'$, where $s>1$ is such that $uv\in\mathrm{RH}_s$. Then,
\begin{align*}
I_{2,2}&\leq\s{j}{}{}\frac{C}{t}\frac{|Q_j|}{v(Q_j)}\left(\frac{1}{|Q_j|}\int_{Q_j}|b(x)-b_{Q_j}|^{ms'}\right)^{1/s'}\left(\frac{1}{|Q_j|}\int_{Q_j}(u(x)v(x))^s\right)^{1/s}\int_{Q_j}f(x)v(x)\,dx\\
&\leq\frac{C}{t}\s{j}{}{\frac{(uv)(Q_j)}{v(Q_j)}\int_{Q_j}f(x)v(x)\,dx}\\
&\leq\frac{C}{t}\int_{\R^n}f(x)u(x)v(x)\,dx,
\end{align*}
by applying again Lemma~\ref{lema: fundamental}.
It only remains the estimation of $I_3$.
\begin{align*}
I_3&\leq\s{i=1}{m-1}{uv\left(\left\{x\in \R^n\backslash\Omega^*: \left|\frac{T_b^i(\s{j}{}{(b-b_{Q_j})^{m-i}h_jv})(x)}{v(x)}\right|>\frac{t}{C}\right\}\right)}\\
&\leq\s{i=1}{m-1}{uv\left(\left\{x\in \R^n\backslash\Omega^*: \left|\frac{T_b^i(\s{j}{}{(b-b_{Q_j})^{m-i}f\mathcal{X}_{Q_j}v})(x)}{v(x)}\right|>\frac{t}{C}\right\}\right)}\\
&+\s{i=1}{m-1}{uv\left(\left\{x\in \R^n\backslash\Omega^*: \left|\frac{T_b^i(\s{j}{}{(b-b_{Q_j})^{m-i}f_{Q_j}^{v}\mathcal{X}_{Q_j}v})(x)}{v(x)}\right|>\frac{t}{C}\right\}\right)}\\
&=I_{3,1}+I_{3,2}.
\end{align*}
In order to estimate $I_{3,1}$ and $I_{3,2}$ we will use the inductive hypothesis and the following fact:
if $A$, $B$ and $C$ are Young functions satisfying
$A^{-1}(t)B^{-1}(t)\lesssim C^{-1}(t)$ for every $t>0$, it is easy to see that
\begin{equation*} %\label{eq: inversas}
C(st)\leq A(s)+B(t),
\end{equation*} holds for every $0 \leq s,t<\infty$.

Let us take $\alpha$ such that $\alpha s'<C_2$, where $C_2$ is the constant that appears in \eqref{eq: J-N} for $b$ and $s'$ is the conjugated exponent of $s$, that verifies $uv\in\mathrm{RH}_s$. Thus, if $\Psi_k(t)=e^{\alpha t^{1/k}}-1$ then
$\Psi_k^{-1}(t)\approx (\log (e+t))^k$. Since $\Phi_m^{-1}(t)\approx t/(\log (e+t))^m$, then $\Phi_m^{-1}(t) \Psi_{m-i}^{-1}(t)\lesssim \Phi_i^{-1}(t)$, so that
by the inductive hypothesis  we obtain  %and \eqref{eq: inversas}
\begin{align*}
I_{3,1}&\leq C\s{i=1}{m-1}{\int_{\R^n}\Phi_i\left(\frac{\s{j}{}{(b(x)-b_{Q_j})^{m-i}f(x)\mathcal{X}_{Q_j}(x)}}{t}\right)u(x)v(x)\,dx}\\
&\leq C\s{i=1}{m-1}{}\s{j}{}{\int_{Q_j}\Phi_i\left(\frac{(b(x)-b_{Q_j})^{m-i}f(x)}{t}\right)u(x)v(x)\,dx}\\
&\leq C\s{h=1}{m-1}{\s{j}{}{\int_{Q_j}\left(\Phi_m\left(\frac{f(x)}{t}\right)+\Psi_{m-i}\left((b(x)-b_{Q_j})^{m-i}\right)\right)u(x)v(x)\,dx}}.
\end{align*}
 Now applying H\"{o}lder's inequality, \eqref{eq: J-N} and
 Lemma~\ref{lema: fundamental} we have that
\begin{align*}
\int_{Q_j}\Psi_{m-i}&\left((b(x)-b_{Q_j})^{m-i}\right)u(x)v(x)\,dx\\&\leq\int_{Q_j}e^{\alpha|b(x)-b_{Q_j}|}u(x)v(x)\,dx\\
&\leq\left(\frac{1}{|Q_j|}\int_{Q_j}e^{|b(x)-b_{Q_j}|\alpha s'}\,dx\right)^{1/s'}\left(\frac{1}{|Q_j|}\int_{Q_j}(u(x)v(x))^s\,dx\right)^{1/s}|Q_j|\\
&\leq \left(\int_0^\infty \alpha s'e^{\alpha s' \lambda}C_1e^{-\lambda C_2}\,d\lambda\right)^{1/s'}(uv)(Q_j)\\
&=C(uv)(Q_j)\\
&\leq C\frac{(uv)(Q_j)}{v(Q_j)}v(Q_j)\\
&\leq \frac{C}{t}\int_{Q_j}f(x)u(x)v(x)\,dx.
\end{align*}
Then,
\[I_{3,1}\leq \frac{C}{t}\int_{\R^n}\Phi_m\left(\frac{f(x)}{t}\right)u(x)v(x)\,dx,\]
as desired. In the case of $I_{3,2}$, let us first note that
\begin{align*}
\frac{f_{Q_j}^v}{t}=\frac{1}{t}\frac{1}{v(Q_j)}\int_{Q_j}f(y)v(y)\,dy
\leq\frac{1}{t}\,Ct=C,
\end{align*}
so that $\Phi_m\left(\frac{f_{Q_j}^v}{t}\right)\leq\Phi_m(C)$. Hence
\begin{align*}
I_{3,2}&\leq C\s{i=1}{m-1}{\s{j}{}{\int_{Q_j}\Phi_i\left(\frac{(b(x)-b_{Q_j})^{m-i}f_{Q_j}^v}{t}\right)uv\,dx}}\\
&\leq C\s{i=1}{m-1}{\s{j}{}{\left(\int_{Q_j}\Phi_m\left(\frac{f_{Q_j}^v}{t}\right)u(x)v(x)\,dx+\int_{Q_j}\Psi_{m-i}\left((b(x)-b_{Q_j})^{m-i}\right)uv\,dx\right)}}\\
&\leq C\s{i=1}{m-1}{\s{j}{}{\frac{(uv)(Q_j)}{v(Q_j)}v(Q_j)}}\\
&\leq \frac{C}{t}\int_{\R^n}f(x)uv\,dx,
\end{align*}
which is the desired estimation for $k=m$, and the result is proved.
\end{proof}

\bigskip

For the proof of Theorem~\ref{teo: caso m_phi} concerning to a mixed inequality for $M_{\Phi}$ we shall use the following technical lemma whose proof is given in \cite{O-P}.
\bigskip
\begin{lema}\label{lema_igualdad_de_integral}
 Let $f$ be a positive and locally integrable function. Then for each $\gamma,\lambda>0$ there exists a number $a\in\R^+$ which depends on $f$ and $\lambda$ that satisfies
 \bigskip
\[\left(\int_{|y|\leq a^{\gamma}}f(y)\,dy\right)a^n=\lambda.\]
\end{lema}
\bigskip
We are now in position to prove the mentioned theorem.
\begin{proof}[Proof of Theorem~\ref{teo: caso m_phi}] As in the proof given in \cite{O-P}, we define the sets $G_k=\{x: 2^k<|x|\leq 2^{k+1}\}$, $I_k=\{x: 2^{k-1}<|x|\leq 2^{k+2}\}$, $L_k=\{x: 2^{k+2}<|x|\}$ and $C_k=\{x: |x|\leq 2^{k-1}\}$. Without loss of generality we can write $g=fv$ and we will assume $t=1$ by homogeneity. Then
\begin{align*}
uw\left(\{x\in \R^n: M_{\Phi}(g)(x)>v(x)\}\right)=&\s{k\in\Z}{}{uw\left(\{x\in G_k: M_{\Phi}(g\mathcal X_{I_k})(x)>v(x)\}\right)}\\
&+ \s{k\in\Z}{}{uw\left(\{x\in G_k: M_{\Phi}(g\mathcal X_{L_k})(x)>v(x)\}\right)}\\
&+uw\left(\{x\in \mathbb R^n: M_{\Phi}(g\mathcal X_{C_k})(x)>v(x)\}\right)\\
=& I+II+III.
\end{align*}

We shall begin by estimating $I$. Recalling that $w=1/\Phi(1/v)$ and $v=|x|^\beta$, if $x\in G_k$ we have
\begin{equation*}
\frac{1}{\Phi\left(\frac{1}{2^{(k+1)\beta}}\right)}\leq w(x)<\frac{1}{\Phi\left(\frac{1}{2^{k\beta}}\right)},
\end{equation*}
and also
\begin{equation*}
2^{(k+1)\beta}\leq v(x)<2^{k\beta}.
\end{equation*}

Using these estimates and the weak modular type of $M_{\Phi}$ with weight $u$ (see \cite{K-P-S}) we get
\begin{align*}I\leq&\s{k\in\Z}{}{\frac{1}{\Phi\left(\frac{1}{2^{k\beta}}\right)}u\left(\left\{x\in G_k: M_{\Phi}(g\mathcal X_{I_k})(x)>2^{(k+1)\beta}\right\}\right)}\\
\leq& \s{k\in\Z}{}{\frac{1}{\Phi\left(\frac{1}{2^{k\beta}}\right)}\int_{\R^n}\Phi\left(\frac{g\mathcal X_{I_k}(x)}{2^{(k+1)\beta}}\right)Mu(x)\,dx}\\
\leq& C\s{k\in\Z}{}{\frac{1}{\Phi\left(\frac{1}{2^{k\beta}}\right)} \Phi\left(\frac{1}{2^{k\beta}}\right)\int_{I_k}\Phi(g(x))Mu(x)\,dx}\\
\leq&C\s{k\in\Z}{}{\int_{I_k}\Phi(g(x))Mu(x)\,dx}\\
=&C\int_{\R^n}\Phi(g(x))Mu(x)\,dx,
\end{align*}
where we have used the submultiplicativity of $\Phi$ . This gives the desired estimation for $I$.

In order to estimate $II$, we define, for $x\in G_k$
\[F(x)=C_n\int_{|y|>|x|}\frac{\Phi(g(y))}{|y|^n}\,dy\]
where $C_n=c_n4^n$ and $c_n$ is the measure of the surface area of the unit sphere $S^{n-1}$. Fix $x\in G_k$ and let $B=B(x_0,r)$ be a ball containing $x$. We want to obtain an upper bound for $\norm{g\mathcal X_{L_k}}_{\Phi,B}$. Note that if $y\in L_k\cap B$, since $x\in G_k$ we have that $\frac{|y|}{2}>|x|$, and then
\[2r\geq |y-x|>|y|-|x|>\frac{|y|}{2}.\]
Since $\Phi$ is submultiplicative, this leads to
\begin{align*}
\frac{1}{|B|}\int_B \Phi\left(\frac{g\mathcal X_{L_k}(y)}{(1/\Phi^{-1}(1/F(x)))}\right)\,dy&\leq \frac{1}{|B|}\int_B \Phi(\Phi^{-1}(1/F(x)))\Phi(g\mathcal X_{L_k}(y))\,dy\\
&\leq \frac{1}{F(x)}\frac{1}{|B|}\int_{B\cap L_k} \Phi(g(y))\,dy\\
&\leq \frac{c_n4^n}{F(x)}\int_{|y|>|x|}\frac{\Phi(g(y))}{|y|^n}\,dy\\
%&\leq \frac{1}{F(x)}C\int_{|y|>|x|}\frac{\Phi(g(y))}{|y|^n}\,dy\\
&=\frac{1}{F(x)}F(x)=1.
\end{align*}
Thus, we get that
\[\norm{g\mathcal X_{L_k}}_{\Phi,B}\leq \frac{1}{\Phi^{-1}(1/F(x))}\]
 and  we can proceed as follows
\begin{align*}
II&\leq \s{k\in \Z}{}{}uw\left(\left\{x\in G_k: \frac{1}{\Phi^{-1}(1/F(x))}>v(x)\right\}\right)\\
&=\s{k\in \Z}{}{}uw\left(\left\{x\in G_k: \Phi^{-1}(1/F(x))\}<\frac{1}{v(x)}\right\}\right)\\
&=\s{k\in \Z}{}{}uw\left(\left\{x\in G_k: F(x)>w(x)\right\}\right)\\
&\leq \s{k\in\Z}{}{\frac{1}{\Phi\left(\frac{1}{2^{k\beta}}\right)}}u\left(\left\{x\in G_k: F(x)>\frac{1}{\Phi\left(\frac{1}{2^{(k+1)\beta}}\right)}\right\}\right)\\
&\leq C \int_0^\infty u\left(\{x\in \R^n: F(x)>t\}\right)\,dt\\
&= C\int_{\R^n}F(x)u(x)\,dx\\
&=C\int_{\R^n}\Phi(g(y))\frac{1}{|y|^n}\int_{|y|>|x|}u(x)\,dx\,dy\\
&\leq C\int_{\R^n}\Phi(g(y))Mu(y)\,dy,
\end{align*}
giving the estimation for the second term.

 In order to estimate $III$, we define, for $x\in G_k$
 \[G(x)=\frac{C_n}{|x|^n}\int_{|y|\leq\frac{|x|}{2}}\Phi(g(y))\,dy.\]
 For a fixed  $x\in G_k$, let us take $B=B(x_0,r)$ a ball containing $x$. If  $y\in C_k$, we get $|y|\leq\frac{|x|}{2}$. By following the same arguments as in the estimation of $II$ we obtain that $\norm{g\mathcal X_{C_k}}_{\Phi,B}\leq1/(\Phi^{-1}(1/G(x)))$. Hence
\begin{equation*}
III\leq uw\left(\{x\in \R^n: G(x)>w(x)\}\right).
\end{equation*}
Let $\gamma=n/(-n-r\beta)$. Note that $\gamma>0$ since, by hypothesis, $\beta<-n$. Now applying Lemma~\ref{lema_igualdad_de_integral} with $\gamma$ and $\lambda=1$, there exists $a>0$ which verifies
\begin{equation}\label{integral_phi_de_g_igual_a_uno}
\left(\int_{|y|\leq a^\gamma}\Phi(g(y))\,dy\right)a^n=1.
\end{equation}
Then,
\begin{align*}
uw\left(\{x\in \R^n: G(x)>w(x)\}\right)&=uw\left(\left\{x: |x|\leq a^\gamma, \frac{C_n}{|x|^n}\int_{|y|\leq \frac{|x|}{2}}\Phi(g(y))\,dy>\frac{1}{\Phi\left(\frac{1}{|x|^\beta}\right)} \right\}\right)\\
+&\sum_{k=0}^{\infty}uw\left(\left\{x: 2^ka^\gamma<|x|\leq 2^{k+1}a^\gamma, \frac{C_n}{|x|^n}\int_{|y|\leq \frac{|x|}{2}}\Phi(g(y))\,dy>\frac{1}{\Phi\left(\frac{1}{|x|^\beta}\right)} \right\}\right)\\
&=A+B.
\end{align*}
Note that
\begin{equation}\label{formula_prueba_teo_M_phi}
\left\{x: \frac{C_n}{|x|^n}\int_{|y|\leq a^\gamma}\Phi(g(y))\,dy>\frac{1}{\Phi\left(\frac{1}{|x|^\beta}\right)} \right\}=\left\{x: \frac{\Phi\left(\frac{1}{|x|^\beta}\right)}{|x|^n}>{C_n}^{-1}a^n\right\}.
\end{equation}
If we set $z=|x|^{-\beta}$ then \[|x|^{-n}\Phi\left(\frac{1}{|x|^\beta}\right)=z^{r+n/\beta}(1+\log^+z)^{\delta}=z^{\alpha}(1+\log^+z)^{\delta}=:\varphi(z),\] where $\alpha=r+n/\beta$, which is positive because $\beta<-n$. It can be proved (see for example \cite[Lemma 1.1.27]{TesisTefi}) that there exists a constant $D\geq 1$  such that  \[\frac{1}{D} z^{1/\alpha} (1+\log^+z)^{-\delta/\alpha}\leq\varphi^{-1}(z)\leq D z^{1/\alpha} (1+\log^+z)^{-\delta/\alpha}.\]
 With this in mind, we can write \eqref{formula_prueba_teo_M_phi} as follows
\begin{align*}
\left\{x:\varphi(|x|^{-\beta})>C_n^{-1}a^n\right\}&= \left\{x: |x|^{-\beta}>\varphi^{-1}(C_n^{-1}a^n)\right\}\\
&\subset \left\{x: |x|^{-\beta}>\frac{(C_n^{-1}a^n)^{1/\alpha}}{D(1+\log^+(C_n^{-1}a^n))^{\delta/\alpha}}\right\}\\
& = \left\{x: D\left(\frac{(1+\log^+(C_n^{-1}a^n))^\delta}{C_n^{-1}a^n}\right)^{1/\alpha}>|x|^{\beta}\right\}\\
& =\left\{x: D^{1/\beta}\left(\frac{(1+\log^+(C_n^{-1}a^n))^{\delta}}{C_n^{-1}a^n}\right)^{1/(\alpha\beta)}<|x|\right\}\\
&=\left\{x: D^{1/\beta}\left(\frac{C_n^{-1}}{(1+\log^+(C_n^{-1}a^n))^{\delta}}\right)^{-1/(\alpha\beta)}a^\gamma<|x|\right\}.
\end{align*}
Since  $D\geq 1$, we have that $D^{1/\beta}\left(\frac{C_n^{-1}}{(1+\log^+(C_n^{-1}a^n))^{\delta}}\right)^{-1/(\alpha\beta)}=:C_0<1$. Thus, we get
\begin{align*}
A&\leq uw\left(\{x: C_0a^\gamma<|x|\leq a^\gamma\}\right)\\
&\leq \int_{|x|>C_0a^{\gamma}}u(x)v^r(x)\,dx\\
&=\s{k=1}{\infty}{\int_{C_02^{k-1}a^\gamma\leq|x|<C_02^ka^\gamma}u(x)v^r(x)\,dx}\\
&\leq \s{k=1}{\infty}{\frac{1}{(C_02^{k-1}a^\gamma)^{-r\beta}}\int_{|x|<C_02^ka^\gamma}u(x)\,dx}\\
&=\s{k=1}{\infty}{2^nC_0^{r\beta}2^{(k-1)(n+r\beta)}\int_{|y|\leq a^\gamma}\Phi(g(y))\left(\frac{1}{(2^ka^\gamma)^n}\int_{|x|<2^ka^\gamma}u(x)\,dx\right)\,dy}\\
&\leq C\int_{\R^n}\Phi(g(y))Mu(y)\,dy.
\end{align*}
To finish the proof, it only remains to estimate part $B$.
\begin{align*}
B&\leq \s{k=0}{\infty}{uv^r\left(\{x: 2^ka^{\gamma}<|x|\leq 2^{k+1}a^\gamma\}\right)}\\
&\leq \s{k=0}{\infty}{\frac{1}{(2^ka^\gamma)^{-r\beta}}\int_{|x|\leq 2^{k+1}a^\gamma}u(x)\,dx}\\
&\leq \s{k=0}{\infty}{\frac{(2^{k+1}a^\gamma)^n}{(2^ka^\gamma)^{-r\beta}}\frac{1}{(2^{k+1}a^{\gamma})^n}\int_{|x|\leq 2^{k+1}a^\gamma}u(x)\,dx}\\
&=C\s{k=0}{\infty}{2^n2^{k(n+r\beta)}\int_{|y|\leq a^\gamma}\Phi(g(y))\left(\frac{1}{(2^{k+1}a^\gamma)^n}\int_{|x|\leq 2^{k+1}a^\gamma}u(x)\,dx\right)\,dy}\\
&\leq C\s{k=0}{\infty}{2^{k(n+r\beta)}\int_{\R^n}\Phi(g(y))Mu(y)\,dy}\\
&\leq C\int_{\R^n}\Phi(g(y))Mu(y)\,dy,
\end{align*}
which completes the proof.
\end{proof}

%%%%%%%%%%%%%%%%%%%%%%%%%%%%%%%%%%%%%%%%%%%%%%

\def\cprime{$'$} \def\cprime{$'$}

\end{document}